\documentclass[a4paper,11pt, reqno]{amsart}
\usepackage{amssymb,amsthm,amsmath}
\usepackage[T1]{fontenc}
\usepackage[colorlinks,citecolor=red,pagebackref,hypertexnames=false]{hyperref}

\numberwithin{equation}{section}
\allowdisplaybreaks[2]
\theoremstyle{plain}
\newtheorem{theorem}{Theorem}[section]
\newtheorem{Def}[theorem]{Definition}
\newtheorem{lemma}[theorem]{Lemma}

\newtheorem{proposition}[theorem]{Proposition}

\theoremstyle{definition}

\theoremstyle{remark}
\newtheorem{remark}[theorem]{Remark}

\newtheorem{case[theorem]}{Case}

\def\norm#1.#2.{\lVert#1\rVert_{#2}}

\author{Sunit Ghosh and Jitendriya Swain}

\address{Sunit Ghosh, Department of Mathematics, Indian Institute of Technology Guwahati, India.}
	\email{g.sunit@iitg.ac.in}

	\address{Jitendriya Swain,   Department of Mathematics, Indian Institute of Technology Guwahati, India.}
	\email{jitumath@iitg.ac.in}
	


\date{\today}

\begin{document}
\baselineskip=20pt
\markboth{} {}

\bibliographystyle{amsplain}
\title[On the Schatten exponent in orthonormal Strichartz estimate...]
{On the Schatten exponent in orthonormal Strichartz estimate for the Dunkl operators}


\begin{abstract}
	In \cite{PRA} and \cite{SSM} the orthonormal Strichartz estimates for the Schr\"odinger equation associated with the Dunkl Laplacian and the Dunkl-Hermite operator are obtained. In this article, we prove a necessary condition on the Schatten exponent for the above orthonormal Strichartz estimates, which turns out to be optimal for the Schr\"odinger equations associated with Laplacian and Hermite operator as a special case.
\end{abstract}

\maketitle
\def\BC{{\mathbb C}} \def\BQ{{\mathbb Q}}
\def\BR{{\mathbb R}} \def\BI{{\mathbb I}}
\def\BZ{{\mathbb Z}} \def\BD{{\mathbb D}}
\def\BP{{\mathbb P}} \def\BB{{\mathbb B}}
\def\BS{{\mathbb S}} \def\BH{{\mathbb H}}
\def\BE{{\mathbb E}}
\def\BN{{\mathbb N}}
\def\LP{{W(L^p(\BR^d, \BH), L^q_v)}}
\def\LPN{{W_{\BH}(L^p, L^q_v)}}
\def\LPQ{{W_{\BH}(L^{p'}, L^{q'}_{1/v})}}
\def\L1{{W_{\BH}(L^{\infty}, L^1_w)}}
\def\LB{{L^p(Q_{1/ \beta}, \BH)}}
\def\SP{S^{p,q}_{\tilde{v}}(\BH)}
\def\f{{\bf f}}
\def\h{{\bf h}}
\def\hp{{\bf h'}}
\def\m{{\bf m}}
\def\g{{\bf g}}
\def\ga{{\boldsymbol{\gamma}}}
\vspace{-.6cm}
\section{Introduction}

The orthonormal Strichartz estimate for the usual Laplacian $\Delta$ on $\mathbb{R}^n$ ($n\geq$1) is of the form:
\begin{align}\label{Schatten exponentA}
	\left\|\sum_{j} \lambda_{j }\left|e^{-i t \Delta}f_{j }\right|^{2}\right\|_{L^q (\mathbb{R}, L^p(\mathbb{R}^n))} \leqslant C  \|\lambda\|_{\ell^{r}},
\end{align}
for any orthonormal systems $(f_j)_j$ in $L_{k}^{2}\left(\mathbb{R}^n\right)$ and any $\lambda = (\lambda_{j })_j \in \ell^r(\mathbb{C})$, where $1 \leq r \leq \frac{2p}{p+1}$ with $(p,q) \in [1,\infty ]^2$ satisfies
\begin{align}\label{vd}
	1\leq q < \frac{n+1}{n-1},  \quad \text { and } \quad  \frac{2}{p}+\frac{n}{q}=n.
\end{align}
The estimate (\ref{Schatten exponentA}) was proved by Frank-Lewin-Lieb-Seiringer \cite{frank} and Frank-Sabin \cite{FS} in connection with Fourier restriction theory.  Using a semi-classical arguments based on coherent states the following necessary condition on the Schatten exponent have been proved in \cite{frank}.
\begin{theorem}\cite{frank}.\label{CH3507}
	Let	$n \geq 1$. If $p, q \geq 1$ satisfies the condition (\ref{vd}), then  the estimate (\ref{Schatten exponentA}) fails for all $r > \frac{2p}{p+1}$.
\end{theorem}
\noindent Thus the Schatten exponent $r=\frac{2p}{p+1}$ in (\ref{Schatten exponentA}) is optimal. For the classical Hermite operator  $\mathcal{H} = -\frac{1}{2} (\Delta - \|x\|^2)$, a similar estimate holds:
\begin{align}\label{Schatten exponentB}
	\left\|\sum_{j} \lambda_{j }\left|e^{-i t \mathcal{H}}f_{j }\right|^{2}\right\|_{L^q ((-\frac{\pi}{2},\frac{\pi}{2}), L^p(\mathbb{R}^n))} \leqslant C  \|\lambda\|_{\ell^{r}},
\end{align}
for any orthonormal systems $(f_j)_j$ in $L_{k}^{2}\left(\mathbb{R}^n\right)$ and any $\lambda = (\lambda_{j })_j \in \ell^r(\mathbb{C})$, where $1 \leq r \leq \frac{2p}{p+1}$ with $(p,q) \in [1,\infty ]^2$ satisfies the condition (\ref{vd}). These estimate was proved by Bez-Hong-Lee-Nakamura-Swano \cite{lee} (see also Mondal-Swain \cite{shyam}). As in the previous case the Schatten exponent $r=\frac{2p}{p+1}$ in (\ref{Schatten exponentB}) is optimal:

\begin{theorem}\label{CH3522}\cite{shyam}.
Let $n \geq 1$. If $p,q \geq 1$ satisfies the condition (\ref{vd}), then  the estimate (\ref{Schatten exponentB}) fails for all $r > \frac{2p}{p+1}$.
\end{theorem}

 Inequalities involving the orthonormal system are very useful in the theory for the many body quantum mechanics (see \cite{lieb, liebb}). The orthonormal inequalities with an optimal Schatten exponent plays a significant role to prove the stability of matter \cite{lieeb,lieb,liebb}. It is worth notable that the orthonormal Strichartz inequality (\ref{Schatten exponentA}) with the optimal Schatten exponent $\frac{2p}{p+1}$ was employed crucially to establish well-posedness and scattering theory for the certain Hartree equation in \cite{Lwin1, Lwin2, Lewin}. Therefore it is important to study the nature of the Schatten exponent in orthonormal inequalities. We refer to \cite{lee,BEZ, shyam1,nakamura} for a detailed study on orthonormal inequalities with an optimal Schatten exponent.

Let $G$ be a finite reflection group on $\mathbb{R}^n$ with root system $\mathcal{R}$. For a $G$-invariant real function $k = (k_\alpha)_{\alpha\in\mathcal{R}}$ (multiplicity function) on $\mathcal{R}$, let $\Delta_{k}$ be the Dunkl\hspace{.2pt} Laplacian (see (\ref{zc})) and $\mathcal{H}_k = -\frac{1}{2} (\Delta_{k} - \|x\|^2)$ be the Dunkl-Hermite operator (see (\ref{zv})). Note that for $k \equiv 0$, the operators $\Delta_k$ and $\mathcal{H}_k$ turns out to be the usual Laplacian $\Delta$ and the classical Hermite operator $\mathcal{H}=-\frac{1}{2}(\Delta-\|x\|^2)$ on $\mathbb{R}^n$ respectively. Let $\gamma =\frac{1}{2}\sum\limits_{\alpha\in\mathcal{R}}k(\alpha)$ and the weighted measure on $\mathbb{R}^n$ be $dw_k(x) =\prod\limits_{\alpha\in \mathcal{R}}|\langle \alpha, x\rangle|^{ k_{\alpha}} dx$.
For $1 \leq p < \infty$, the space $L_k^p(\mathbb{R}^n)$ consists of all functions $f$ on $\mathbb{R}^n$ such that $\|f\|_{L_k^p} = \left(c_k\int_{\mathbb{R}^n} |f|^p dw_k\right)^{1/p} < \infty$, where $c_k^{-1} = \int_{\mathbb{R}^n} e^{-\|x\|^2 /2} dw_k(x)$; and $L_k^{\infty}(\mathbb{R}^n)$ is given in the usual way.

 The orthonormal Strichartz estimate (\ref{Schatten exponentA}) is further generalized for the Dunkl Laplacian $\Delta_k$ and the Dunkl-Hermite operator $\mathcal{H}_k$ by Senapati-Pradeep-Mondal-Mejjaoli \cite{PRA} and Mondal-Song \cite{SSM} respectively.
\begin{theorem}\cite{PRA}. (Orthonormal Strichartz estimate for the Dunkl Laplacian)\label{THM1} Suppose $k \geq 0$ and
	$p, q, n \geqslant$ 1 such that
	$$
	1 \leqslant p<\frac{2\gamma+n+ 1}{  2\gamma+n-1},\quad	\frac{2}{q}+\frac{2 \gamma+n}{ p}=2 \gamma+n \quad \text { and } \quad  r = \frac{2 p}{p+1}.
	$$
	Then for any orthonormal system $(f_j)_j$ in $L_{k}^{2}\left(\mathbb{R}^n\right)$ and all sequence $\lambda = (\lambda_{j })_j \in \ell^r(\mathbb{C})$, we have
	\begin{align}\label{Schatten exponent11}
		\left\|\sum_{j} \lambda_{j }\left|e^{-i t \Delta_{k}}f_{j }\right|^{2}\right\|_{L^q (\mathbb{R}, L_{k}^p(\mathbb{R}^n))} \leqslant C  \|\lambda\|_{\ell^r},
	\end{align}
	where $C>0$ is independent of  $(f_j)_j$ and $\lambda$.
\end{theorem}

\begin{theorem} \cite{SSM}. (Orthonormal Strichartz estimate for the Dunkl-Hermite operator)\label{THM2} Suppose $k \geq 0$ and
	$p, q, n \geqslant$ 1 such that
	$$
	1 \leqslant p<\frac{2\gamma+n+ 1}{  2\gamma+n-1},\quad	\frac{2}{q}+\frac{2 \gamma+n}{ p}=2 \gamma+n \quad \text { and } \quad  r = \frac{2 p}{p+1}.
	$$
	Then for any orthonormal system $(f_j)_j$ in $L_{k}^{2}\left(\mathbb{R}^n\right)$ and all sequence $\lambda = (\lambda_{j })_j \in \ell^r(\mathbb{C})$, we have
	\begin{align}\label{Schatten exponent1}
		\left\|\sum_{j} \lambda_{j }\left|e^{-i t \mathcal{H}_{k}}f_{j }\right|^{2}\right\|_{L^q ((-\frac{\pi}{2},\frac{\pi}{2}), L_{k}^p(\mathbb{R}^n))} \leqslant C  \|\lambda\|_{\ell^r},
	\end{align}
	where $C>0$ is independent of  $(f_j)_j$ and $\lambda$.
\end{theorem}
In view of the inclusion relation of $\ell^r(\mathbb{C})$, the Theorems \ref{THM1} and \ref{THM2} holds true for all $1 \leq r \leq \frac{2p}{p+1}$. The validity of such results is not known for $r \geq \frac{2p}{p+1}$ in the literature so far. In this article, we give a necessary condition on  $r$ for which the Theorems \ref{THM1} and \ref{THM2} can hold.

We introduce a suitable set of coherent states in Dunkl setting that fits in our context and using a semi-classical argument we obtain the following.
\begin{theorem}\label{Thm1}(Necessary condition on Schatten exponent)
	Suppose $k \geq 0$ and	$p, q, n \geqslant$ 1 satisfies
	$$ 2 \gamma < \frac{n(p+1)}{p-1} \quad \text { and } \quad r>\frac{2 p n}{(p+1)n - (p-1)2\gamma}.$$ Then there exists an orthonormal system $(f_j)_j$ in $L_{k}^{2}\left(\mathbb{R}^n\right)$ and a sequences $\lambda = (\lambda_{j })_j \in \ell^r(\mathbb{C})$ such that
	\begin{equation}\label{uy}
		\left\|\sum_{j} \lambda_{j }\left|e^{-i t \mathcal{H}_{k}}f_{j }\right|^{2}\right\|_{L^q ((-\frac{\pi}{2},\frac{\pi}{2}), L_{k}^p(\mathbb{R}^n))} = \infty.
	\end{equation}
\end{theorem}
\noindent Note that if $p, q, n \geq 1$ satisfy $\frac{2}{q}+\frac{2 \gamma+n}{ p}=2 \gamma+n$, then  $ \frac{(p+1)n - (p-1)2\gamma}{2pn} = 1-\frac{1}{nq} > 0$, thus for such pair $p,q $ in Theorem \ref{THM2} the estimate (\ref{Schatten exponent1}) fails for all $r>\frac{2 p n}{(p+1)n - (p-1)2\gamma}$.

Using the kernel relation between the semigroups $e^{-i t \mathcal{H}_{k}}$ and $e^{i t \Delta_{k}}$, we obtain:
\begin{theorem}\label{thm2}(Necessary condition on Schatten exponent)
	Suppose $k \geq 0$ and	$p, q, n \geqslant$ 1 satisfies
	$$ 	\frac{2}{q}+\frac{2 \gamma+n}{ p}=2 \gamma+n \quad \text { and } \quad r>\frac{2 p n}{(p+1)n - (p-1)2\gamma}.$$ Then there exists an orthonormal system $(f_j)_j$ in $L_{k}^{2}\left(\mathbb{R}^n\right)$ and a sequences $\lambda = (\lambda_{j })_j \in \ell^r(\mathbb{C})$ such that
	\begin{equation}
		\left\|\sum_{j} \lambda_{j }\left|e^{- i t \Delta_{k}}f_{j }\right|^{2}\right\|_{L^q (\mathbb{R}, L_{k}^p(\mathbb{R}^n))} = \infty.
	\end{equation}
\end{theorem}
 \noindent In view of Theorem \ref{thm2}, we conclude for such pair $p,q $ in Theorem \ref{THM1} the estimate (\ref{Schatten exponent11}) fails for all $r>\frac{2 p n}{(p+1)n - (p-1)2\gamma}$.

 Further, using the inclusion relation of $L^q(-\frac{\pi}{2},\frac{\pi}{2})$- spaces, Theorem \ref{THM2} can be extended to a wider range of $p,q$ (also generalizes Theorem A of Ben Sa\"id-Nandakumaran-Ratnakumar \cite{Ratna3} to orthonormal systems) in the following.
 \begin{theorem}\label{thm3}
 	 Suppose $k \geq 0$ and
 	$p, q, n \geqslant$ 1 such that $(p,q)=(\infty,1)$ or
 	$$
 	1 <  p<\frac{2\gamma+n+ 1}{  2\gamma+n-1}\quad \text { and } \quad	\frac{2}{q}+\frac{2 \gamma+n}{ p} \geq 2 \gamma+n.
 	$$
 	Then the estimate (\ref{Schatten exponent1}) holds for any orthonormal system $(f_j)_j$ in $L_{k}^{2}\left(\mathbb{R}^n\right)$ and all sequence $\lambda = (\lambda_{j })_j \in \ell^r(\mathbb{C})$, if $1 \leq r \leq \frac{2p}{p+1}$; and fails for all $r>\frac{2 p n}{(p+1)n - (p-1)2\gamma}$.
 \end{theorem}

Throughout this paper we shall use the standard multi-index notations.
For multi-indices $\nu\in \mathbb{N}^n_0$, we write $|\nu| = \sum_{j=1}^{n} \nu_j$ as well as $z^\nu = \Pi_{j=1}^{d}z_j^{\nu_j}, D^\nu = \Pi_{j=1}^{d} D_j^{\nu_j}$, for $z \in \mathbb{C}^n$ and any family $D=(D_1,\cdots,D_n)$ of commuting operators.

The paper is organized as follows: In Section \ref{sec2}, we recall some basic definitions and important properties of Dunkl operators and Dunkl kernel, and Schr\"odinger semigroups associated to Dunkl operators. In Section \ref{sec3}, we introduce the coherent states in Dunkl setting and discuss some of its properties that helps to prove our main results. In Section \ref{sec4} we prove Theorems \ref{Thm1}, \ref{thm2} and \ref{thm3}.

\section{Preliminaries}\label{sec2}
In this section we discuss some basic definitions of Dunkl operators and Dunkl kernel, and provide necessary background information  about generalized Fock space and  Schr\"odinger semigroups associated to Dunkl operators.

\subsection{Root systems, reflection groups and multiplicity functions }\label{dunklop} Consider $x =(x_1,...,x_n) \in \mathbb{R}^{n}$. The scalar product of $x,y\in\mathbb{R}^n$ is denoted by $\langle x, y\rangle := \sum_{j=1}^{n} x_j y_j$  and the norm of $x$ is denoted by $\|x\|=\langle x, x\rangle^{1 / 2}$. On $\mathbb{C}^n$, $\|\cdot\|$ also denotes  the standard Hermitian norm, while $\langle z,w\rangle = \sum_{j=1}^{n} z_j w_j$ and $\ell(z) = \langle z,z \rangle$.

For $\alpha \in \mathbb{R}^{n} \backslash\{0\},$ we denote $r_{\alpha}$ as  the reflection with respect to the hyperplane $\langle \alpha\rangle^{\perp}$ orthogonal to $\alpha$ and is defined by
$$
r_{\alpha}(x):=x-2 \frac{\langle \alpha, x\rangle}{\|\alpha\|^{2}} \alpha, \quad x \in \mathbb{R}^{n}.
$$
A finite set $\mathcal{R}$ in $\mathbb{R}^{n} \backslash\{0\}$ is  said to be a  root system if $ r_{\alpha}(\mathcal{R})=\mathcal{R}$ and
$\mathcal{R} \cap \mathbb{R} a=\{\pm \alpha\}$ for all $\alpha \in \mathcal{R}$.
We assume that it is normalized by $\|\alpha\|^2 = 2$ for all $\alpha \in \mathcal{R}$. For a given root system $\mathcal{R}$  the reflections $\left\{r_{\alpha} \mid \alpha \in \mathcal{R}\right\}$ generate a finite group $G \subset O(n, \mathbb{R})$, known as the finite Coxeter group associated with $\mathcal{R}$. For a   detailed on the theory of finite reflection groups, we refer to  \cite{hum}.  	Let    $\mathcal{R}^+:=\{\alpha\in\mathcal{R}:\langle\alpha,\beta\rangle>0\}$ for some $\beta\in\mathbb{R}^n\backslash\bigcup_{\alpha\in\mathcal{R}}\langle \alpha\rangle^{\perp}$,  be a  fix  positive root system.


$A$ multiplicity function for $G$ is a function $k: \mathcal{R} \rightarrow \mathbb{C}$ which is constant on $G$-orbits. Setting $k_{\alpha}:=k(\alpha)$ for $\alpha \in \mathcal{R},$  from the definition of $G$-invariant,  we have $k_{g \alpha}=k_{\alpha}$ for all $g \in G$.  We say $k$ is non-negative if $k_{\alpha} \geq 0$ for all $\alpha \in \mathcal{R}$. The $\mathbb{C}$-vector space of non-negative multiplicity functions on $\mathcal{R}$ is denoted by $\mathcal{K}^{+}$.
For $k \in \mathcal{K}^{+}$, let  $$\gamma :=\frac{1}{2}\sum\limits_{\alpha\in\mathcal{R}}k(\alpha) = \sum\limits_{\alpha\in\mathcal{R}^+}k(\alpha)$$ and the associated measure on $\mathbb{R}^n$ be
\begin{align*}
	dw_k(x) :=\prod\limits_{\alpha\in \mathcal{R}}|\langle \alpha, x\rangle|^{ k_{\alpha}} dx=\prod\limits_{\alpha\in \mathcal{R}^{+}}|\langle \alpha, x\rangle|^{2 k_{\alpha}} dx.
\end{align*}
The volume $ w_k(B(x, r))$ of the ball $B(x, r) := \{\xi \in \mathbb{R}^n : \|x - \xi\| < r \}$, centred at $x$ and radius $r$, satisfies
\begin{align}\label{j3}
	w_k(B(\delta x, \delta r)) = \delta^{2\gamma + n} w_k(B(x,r))
\end{align}
for all $x \in \mathbb{R}^n, \delta, r >0$ (see \cite{Ank}).
Observe that there is a constant $C > 0$ such that for all $x \in \mathbb{R}^n$ and $r > 0$ we have
\begin{align}\label{j2}
	C^{-1} w_k(B( x,  r)) \leq r^{n} \Pi_{\alpha \in \mathcal{R}^{+}}(| \langle \alpha,x \rangle| + r)^{2k(\alpha)} \leq C w_k(B(x,r)).
\end{align}

\subsection{Dunkl operators and Dunkl kernel} For $k \in \mathcal{K}^{+}$, Dunkl in 1989 introduced a family of first order differential-difference operators $T_j (j=1,...,n)$, called the Dunkl operators, by
\begin{align}\label{dunkl}
	T_{j} f(x):=\partial_{j} f(x)+\sum_{\alpha \in \mathcal{R}^{+}} k_{\alpha }\alpha_j \frac{f(x)-f\left(r_{\alpha} x\right)}{\langle \alpha, x\rangle}, \quad f \in C^{1}\left(\mathbb{R}^{n}\right),
\end{align}
where $\partial_{j}$ denotes the $j$th partial derivative. When $k \equiv 0$, $T_j$'s reduce to the corresponding partial derivatives. Let $\mathcal{P} = \mathbb{C}[\mathbb{R}^n]$ be the algebra of polynomial functions on $\mathbb{R}^n$ and $\mathcal{P}_l,\hspace{2pt} l \in \mathbb{Z_{+}} = \{0,1,2,\cdots\}$, the subspace of homogeneous polynomials of degree $l$. Then the set $\{T_j\}$ defines a commutative algebra of differential-difference operators on $\mathcal{P}$ and each $T_j$ is homogeneous of degree $-1$ on $\mathcal{P}$, that is, $T_j \mathcal{P}_l \subset \mathcal{P}_{l-1}$.

An analog of the Laplace operator, the Dunkl Laplacian is defined by
$\Delta_k=\sum_{j=1}^nT_{j}^2.$ It is a second-order differential-difference operator and given explicitly by
\begin{align}\label{zc}
	\Delta_kf(x)=\Delta f(x)+\sum_{\alpha \in \mathcal{R}^+}k_\alpha \left\{\frac{2 \langle \nabla f(x),  \alpha \rangle}{\langle \alpha, x\rangle } -\|\alpha\|^2 \frac{f(x)-f(r_\alpha x)}{\langle \alpha, x\rangle^2 } \right\}, \quad f\in C^1(\mathbb{R}^n),
\end{align}
where $\nabla$ and $\Delta$ are      the usual gradient  and usual Laplacian operator on $\mathbb{R}^n$, respectively. Associated to this operator we consider the Dunkl-Hermite operator
\begin{align}\label{zv}
	\mathcal{H}_k = -\frac{1}{2} (\Delta_{k} - \|x\|^2).
\end{align}
For $k \equiv 0$, the operators $\Delta_k$ and $\mathcal{H}_k$ turns out to the classical Laplacian $\Delta$ and the classical Hermite operator $-\frac{1}{2} (\Delta - \|x\|^2)$ on $\mathbb{R}^n$ respectively.

%


For $y \in \mathbb{R}^n$, the initial value problem
\begin{align}
	\left\{  \begin{array}{ll} T_j^x u(x,y)  = z_j u(x,y), \quad j=1,\cdots,n, \\ u(0,y) =1; \end{array}\right.
\end{align}
admite a unique analytic solution on $\mathbb{R}^n$, denoted by $E_k(x,y)$ and called the Dunkl kernel. This kernel has a unique holomorphic extension to $\mathbb{C}^n \times \mathbb{C}^n$.

The Dunkl kernel possesses the following properties: for $z, w \in \mathbb{C}^n$, $\beta \in \mathbb{C}$,
\begin{align}
	E_k(z,w) = E_k(w,z),& \hspace{5pt}
	E_k(\beta z, w) = E_k(z,\beta w),\label{1B}\\
  \overline{E_k(z,w)} = E_k(\bar{z}, \bar{w}),&\hspace{5pt} |E_k(z,w)| \leq e^{\|z\|\hspace{2pt}\|w\|}.\label{1C}
\end{align}
Moreover, the Dunkl kernel $E_k(x,y)$ satisfies the following upper and lower bound estimate (see \cite{Ank}):
for every $\epsilon > 0$, there exists $C \geq 1$ such that
\begin{align}\label{j1}
	\frac{C^{-1}}{w_k(B(x,1))} e^{\frac{\|x\|^2+\|y\|^2}{2}} \leq E_k(x,y) \leq \frac{C}{w_k(B(x,1))} e^{\frac{\|x\|^2+\|y\|^2}{2}},
\end{align} for all $x, y \in \mathbb{R}^n$ satisfying $\|x-y\| < \epsilon$.

The following proposition is crucial in Dunkl's theory and its applications.
\begin{proposition}\cite{Dunklll}
	For $k \geq 0$, $z, w \in \mathbb{C}^n$
	\begin{equation}\label{1K}
		\int_{\mathbb{R}^n} e^{-\|x\|^2/2} E_k(x,z) E_k(x,w) w_k(x)dx = c_k e^{(\ell(z) + \ell(w))/2} E_k(z,w).
	\end{equation}
\end{proposition}

\subsection{Generalized Hermite polynomials} Consider the following bilinear form on $\mathcal{P}$: $$[p,q]_k := \left(p(T)q\right)(0),\quad p,q \in \mathcal{P},$$ here $p(T)$ is the operator derived from $p(x)$ by replacing $x_j$ by $T_j$. Due to Dunkl \cite{dun}, the pairing $[\cdot,\cdot]_k$ is in fact a scalar product on $\mathcal{P}$ and $[p, q]_k = 0$, for $p \in \mathcal{P}_l, q \in \mathcal{P}_m$ with $l \ne m$.

Let $\{\varphi_\nu : \nu \in \mathbb{N}^n_0 \}$ be an orthonormal basis of $\mathcal{P}$ with respect to the scalar
product $[\cdot,\cdot]_k$ such that $\varphi_\nu \in \mathcal{P}_{|\nu|}$ and the coefficients of the $\varphi_\nu$ are real. As $\mathcal{P} = \oplus_{l \in \mathbb{N}_0} \mathcal{P}_l$ and $\mathcal{P}_l \perp \mathcal{P}_m$ for $l \ne m$, the $\varphi_\nu$ with $|\nu| = l$ can for example be constructed by Gram-Schmidt orthogonalization within $\mathcal{P}_l$ from
an arbitrary ordered real-coefficients basis of $\mathcal{P}_l$.\\
\noindent As in the classical case, R\"osler\cite{Rost} obtained the following connection of the basis $\varphi_\nu$ with the Dunkl kernel:
\begin{align}\label{nb}
	E_k(z,w) = \sum_{\nu \in \mathbb{N}^n_0}\varphi_\nu(z)\varphi_\nu(w); \quad z,w \in \mathbb{C}^n,
\end{align}
where the convergence is locally uniform on $\mathbb{C}^n \times \mathbb{C}^n$.

The generalized Hermite polynomials $\{H_\nu : \nu \in \mathbb{N}^n_0\}$ and Hermite functions $\{h_\nu : \nu \in \mathbb{N}^n_0 \}$, associated with the basis $\{\varphi_\nu : \nu \in \mathbb{N}^n_0 \}$, is defined by
\begin{equation}
	H_\nu(x) := 2^{|\nu|} e^{- {\Delta_k/4}} \varphi_\nu = 2^{|\nu|} \sum_{l=0}^{[{|\nu|/2}]} \frac{(-1)^l}{2^{2l} l!} \Delta_k^l \varphi_\nu(x),
\end{equation}
and
\begin{equation}
	h_\nu(x) := 2^{-|\nu|/2} e^{-\|x\|^2/2} H_\nu(x), \quad x \in \mathbb{R}^n.
\end{equation}
 Now we recall some properties of the generalized Hermite polynomials and Hermite functions in the following  proposition (see \cite{Rost}, p. 525-531).
\begin{proposition}
	\begin{enumerate}
		\item  The generalized Hermite functions $\{h_\nu : \nu \in \mathbb{N}^n_0 \}$ form a complete set of eigenfunctions for the Dunkl-Hermite operator $\mathcal{H}_k$ with $\mathcal{H}_k h_\nu = (2|\nu| + 2\gamma + n) h_\nu.$
		\item The set $\{h_\nu : \nu \in \mathbb{N}^n_0 \}$ is an orthonormal basis of $L^2_k(\mathbb{R}^n)$.
		\item For all $z,w \in \mathbb{C}^n$, there is a generating function for the generalized Hermite polynomials,
		\begin{equation}
			e^{-\ell(w)} E_k(2z,w) = \sum_{\nu \in \mathbb{N}^n_0} H_\nu(z)\varphi_\nu(w).\label{1D}
		\end{equation}
		\item For $r \in \mathbb{C}$ with $|r| < 1$ and for all $x,y \in \mathbb{R}^n$, the Mehler-formula for the generalized Hermite polynomials,
		\begin{equation}\label{1E}
			\sum_{\nu \in \mathbb{N}^n_0} \frac{H_\nu(x) H_\nu(y)}{2^{|\nu|}} r^{|\nu|} = \frac{1}{(1-r^2)^{\gamma + n/2}} e^{- \frac{r^2 (x^2 + y^2)}{1-r^2}} E_k\left(\frac{2rx}{1-r^2}, y \right).
		\end{equation}
	\end{enumerate}
\end{proposition}

\subsection{Generalized Fock space} In \cite{Soltani}, Soltani introduced (also introduced around the same time independently by Ben Sa\"id-\O{}rsted, see \cite{BBen}) the generalized Fock space $\mathcal{A}_k$ associated with the Dunkl operators. This is a Hilbert space of holomorphic functions on $\mathbb{C}^n$ with the reproducing kernel $E_k(z,\bar{w})$, for $z, w \in \mathbb{C}^n$. More precisely,
$$\mathcal{A}_k = \{ f(z) = \sum_{\nu \in \mathbb{N}^n_0} a_\nu \varphi_\nu(z) : \|f\|^2_k := \sum_{\nu \mathbb{N}_0^n} |a_\nu|^2 < \infty \}.$$
For $f(z) = \sum_{\nu \in \mathbb{N}^n_0} a_\nu \varphi_\nu(z) , g(z) = \sum_{\nu \in \mathbb{N}^n_0} b_\nu \varphi_\nu(z) \in \mathcal{A}_k$ the inner product in $\mathcal{A}_k$ is given by $(f,g)_k = \sum_{\nu \in \mathbb{N}^2_0} a_\nu \bar{b_\nu}$, .

For $z, w \in \mathbb{C}^n$, let $U_k(z,w) = e^{-(\ell(z) + \ell(w))/2} E_k(\sqrt{2}z,w)$. The chaotic transform (also called as generalized Segal-Bargmann transform) is the transformation defined on $L^2_k(\mathbb{R}^n)$, by
\begin{equation}
	\mathcal{C}_k(f)(z) := \int_{\mathbb{R}^n} U_k(z,x) f(x) w_k(x)dx.
\end{equation}
	The chaotic transform $\mathcal{C}_k$ is a unitary mapping of $L^2_k(\mathbb{R}^n)$ onto $\mathcal{A}_k$. Moreover, the basis elements are related by
	\begin{equation}\label{1I}
		\mathcal{C}_k (h_\nu) = \varphi_\nu.
	\end{equation}

\subsection{Schr\"odinger semigroup associated to Dunkl-Hermite and Dunkl Laplacian}
Consider the initial value problem (IVP) for the Dunkl-Hermite-Schr\"odinger equation
\begin{align}\label{1001}
	\left\{  \begin{array}{ll}i \partial_{t} u(t, x)- \mathcal{H}_{k} u(t, x)=0, & (t, x) \in \mathbb{R} \times  \mathbb{R}^n, \\ u(0,x)=f(x) \end{array}\right.
\end{align}
 If  $f\in L^2_k(\mathbb{R}^n)$, the solution of the IVP (\ref{1001}) is given  by $u(t,x)=e^{-i t \mathcal{H}_k} f(x).$ In \cite{ben}, the authors have studied the Dunkl-Hermite semigroup ($a = 2$ case) $\mathcal{I}_{k}(z)$ with infinitesimal generator $\mathcal{H}_k,$ that is
$
I_{k}(z):=e^{-z \mathcal{H}_k}
$
for $z \in \mathbb{C}$ such that $\operatorname{Re}(z) \geq 0$. In particular for all $t \in \mathbb{R}$,
$e^{-i t \mathcal{H}_k}$ is an unitary operator on $L^2_k(\mathbb{R}^n)$. For $t \in \mathbb{R}\setminus \pi\mathbb{Z}$, the operator $e^{-i t \mathcal{H}_k}$ can be expressed as
\begin{equation}\label{semigroup}
		e^{-i t \mathcal{H}_k} f(x)=c_{k} \int_{\mathbb{R}^{n}} \Lambda_{k}(x, y ; t) f(y) w_{k}(y) d y,
\end{equation}
with the kernel
\begin{align}
	\Lambda_{k}(x, y ; t) = \frac{1}{(i \sin t)^{{ \gamma+\frac{n}{2}}}} e^ {-\frac{i}{2}\operatorname{cot} t \left(\|x\|^{2}+\|y\|^{2}\right)} E_k \left( \frac{ix}{\sin t} , y \right).
\end{align}

Now, consider the IVP for the Dunkl-Schr\"odinger equation
\begin{align}\label{100}
	\left\{  \begin{array}{ll}i \partial_{t} u(t, x)- \frac{1}{2}\Delta_{k} u(t, x)=0, & (t, x) \in \mathbb{R} \times  \mathbb{R}^n, \\ u(0,x)=f(x) \end{array}\right.
\end{align}
If $f \in L^2_k(\mathbb{R}^n)$ the solution of the IVP (\ref{100}) is given by
\begin{equation}\label{semigroup}
	u(t,x) = e^{-i \frac{t}{2} \Delta_k} f(x)=c_{k} \int_{\mathbb{R}^{n}} \Gamma_{k}(x, y ; t) f(y) w_{k}(y) d y,
\end{equation}
where
\begin{align}
	\Gamma_{k}(x, y ; t) = \frac{1}{(i t)^{\gamma+\frac{n}{2}}} e^ {-\frac{it}{2} \left(\|x\|^{2}+\|y\|^{2}\right)} E_k \left( \frac{ix}{ t} , y \right).
\end{align}
 Applying the change of variable $s=\tan (t)$ with $t \in(-\pi / 2, \pi / 2),$ we get
\begin{align}\label{gf}
	\Lambda_{k}(x, y ; i \tan^{-1} s)&=c_{k}^{-1}\left(1+s^{2}\right)^{\frac{2 \gamma+n}{4}} \exp \left(-i s \frac{\|x\|^{2}}{2}\right) \Gamma_{k}\left(\left(1+s^{2}\right)^{\frac{1}{2}} x, y ; i s\right).
\end{align}
Using kernel relation (\ref{gf}) the following relation can be obtained (see the proof of Theorem 6.1 in \cite{SSM}).
\begin{lemma}\label{ed}
	Suppose $n \geq 1$. If $p, q \geq 1$ satisfy
	$
	\frac{1}{q}+\frac{2 \gamma+n}{2p}=\frac{2 \gamma+n}{2},
	$ then
		\begin{align}\label{Schatten exponent}
		\left\|\sum_{  j } n_{j }\left|e^{i t \Delta_{k}}f_{j }\right|^{2}\right\|_{L^q (\mathbb{R}, L_{k}^p(\mathbb{R}^n))} = \left\|\sum_{  j } \lambda_{j }\left|e^{-i t \mathcal{H}_{k}}f_{j }\right|^{2}\right\|_{L^q ((-\frac{\pi}{2},\frac{\pi}{2}), L_{k}^p(\mathbb{R}^n))} ,
	\end{align}
for any system $\left\{f_{\iota}\right\}$ of orthonormal functions in $L_{k}^{2}\left(\mathbb{R}^n\right)$ and any coefficients $\left\{\lambda_{j}\right\}$ in $\mathbb{C}$.
\end{lemma}

\subsection{Schatten class}\label{subsec2.2}
Let ${H}$ be a complex and separable Hilbert space equipped with the inner product is denoted by $\langle, \rangle_{H}$. Let $T:{H} \rightarrow {H}$ be a compact operator and let    $T^{*}$ denotes the adjoint of $T$.  For $1 \leq r<\infty,$ the Schatten space $\mathcal{G}^{r}({H})$ is defined as the space of all compact operators $T$ on ${H}$ such that $$\sum_{n =1}^\infty  \left(s_{n}(T)\right)^{r}<\infty,$$ where $s_{n}(T)$ denotes the singular values of \(T,\) i.e., the eigenvalues of \(|T|=\sqrt{T^{*} T}\) counted according to multiplicity. For $T\in \mathcal{G}^{r}({H})$, the Schatten $r$-norm is defined by  $$\|T\|_{\mathcal{G}^{r}}=\left(\sum_{n=1}^{\infty}\left(s_{n}(T)\right)^{r}\right)^{\frac{1}{r}}.$$

An operator belongs to the class \(\mathcal{G}^{1}({{H}})\) is known as {\it Trace class} operator. Also, an operator belongs to   \(\mathcal{G}^{2}({{H}})\) is known as  {\it Hilbert-Schmidt} operator. For $r=\infty $ the class \(\mathcal{G}^{\infty}({{H}})\)  is the collection of all bounded linear operators on $H$ equipped with the operator norm.

\section{Coherent states and the Dunkl-Hermite Scro\"odinger semigroup}\label{sec3}
In this section we construct a parameterized family of functions on $L^2_k(\mathbb{R}^n)$ that plays the role of coherent states and study its properties.
\subsection{Coherent states}
\begin{Def}
	A set of coherent states in a Hilbert space $H$ is a subset $\{\Phi_x\}_{x\in X}$ of $H$ such that
	\begin{enumerate}
		\item $X$ is a locally compact topological space and the mapping $x \mapsto \Phi_x : X \rightarrow H$ is continuous.
		\item  there is a positive Borel measure $d\mu$ on $X$ such that, for $f \in H$, we have  $$\int_X |\langle \Phi_x,f\rangle_H|^2 d\mu(x) = \|f\|^2_H.$$
	\end{enumerate}
\end{Def}
For $z \in \mathbb{C}^n$, let
\begin{equation}\label{mb}
	F_z(x) = U_k(z,x) = e^{-(\ell(z) + \ell(x))/2} E_k(\sqrt{2}z,x), \quad x\in \mathbb{R}^n.
\end{equation}
It is easy to check that $F_z \in L^2_k(\mathbb{R}^n)$ with $\|F_z\|_{L^2_k} = \sqrt{E_k(z,\bar{z})}$ (see Lemma 3 in \cite{Soltani}). Since the chaotic transform $\mathcal{C}_k : L^2_k(\mathbb{R}^n) \rightarrow \mathcal{A}_k$ is  unitary, i. e.,
$ \|\mathcal{C}_k(f)\|_k^2 = \|f\|^2_{L^2_k}$,  the bilinear extension can be obtained by polarization in the following:
\begin{equation}\label{1A}
	(\mathcal{C}_k(f),\mathcal{C}_k(g))_k = \int_{\mathbb{R}^n} f(x)\overline{g(x)} dw_k(x),
\end{equation}
for all $f,g \in L^2_k(\mathbb{R}^n)$.

Consider the linear operator $\gamma_1$ on $L^2_k(\mathbb{R}^n)$ defined by
\begin{equation}
	\gamma_1(f) (x) := (\mathcal{C}_k(f)(z), F_z(x))_k = \left(
	\int_{\mathbb{R}^n} F_z(y)f(y) dw_k(y),F_z(x)\right)_k.
\end{equation}
Now (\ref{1A}) shows that $\gamma_1 = \boldsymbol{1}$ weakly.

 For $k\equiv 0$ and $n \geq 1$, the inner product on $\mathcal{A}_0$ is given by $(\phi,\psi)_0=\frac{1}{\pi^n} \int_{\mathbb{C}^n}\phi(z)\overline{\psi(z)}e^{-|z|^2}dz$, for $\phi,\psi \in \mathcal{A}_0$. In this case the coherent states are given by the family $\{F_z\}_{z \in \mathbb{C}^n}$ (defined in (\ref{mb}))  in $L^2(\mathbb{R}^n)$ (see also \cite{Simon}). Further,
for $k \geq 0$ and $n=1$, the inner product on $\mathcal{A}_k$ is given by $(\phi,\psi)_k = \int_{\mathbb{C}}\phi(z)\overline{\psi(z)}d\mu_k(z)$, for $\phi,\psi \in \mathcal{A}_k$, where the measure $d\mu_k$ can be found explicitly in \cite{BBen}, and from (\ref{1A}) the family $\{F_z\}_{z \in \mathbb{C}}$ defines coherent states in $L_k^2(\mathbb{R})$. We also refer to Ghazouani\cite{GG} for the study of one dimensional coherent states in Dunkl setting. However, the existence of the positive Borel measure $d\mu_k$ such that $(\phi,\psi)_k = \int_{\mathbb{C}^n}\phi(z)\overline{\psi(z)}d\mu_k(z)$ is not known for $n\geq2$. (See the discussion below Lemma 3.9 in \cite{BBen}). In view of (\ref{1A}) the family $\{F_z\}_{z \in \mathbb{C}}$ plays the role of coherent states in $L_k^2(\mathbb{R}^n)$.


\begin{lemma}\label{qa}
	For $w, z \in \mathbb{C}^n$ and $x \in \mathbb{R}^n$, we have
	\begin{equation}
			F_{w z}(x) = \sum_{\nu \in \mathbb{N}^n_0} h_\nu(x) \varphi_\nu(z) {w^{|\nu|}}.
	\end{equation}
\end{lemma}
\begin{proof}
	Using (\ref{1B}) and (\ref{1D}) we can write
	\begin{align}
		F_{w z}(x) = e^{-x^2/2} e^{-\ell(w z)/2} E_k(2x,\frac{w}{\sqrt{2}} z) = e^{-\ell(x)/2} \sum_{\nu \in \mathbb{N}^n_0} H_\nu(x)\varphi_\nu(\frac{w}{\sqrt{2}}z)\nonumber.
	\end{align}
	Since $\varphi_\nu$ is homogeneous of degree $|\nu|$, i. e., $\varphi_\nu(\frac{w}{\sqrt{2}}z) =\frac{w^{|\nu|}}{2^{\frac{|\nu|}{2}}} \varphi_\nu(z)$, we have
	\begin{align*}
		F_{w z}(x) =  \sum_{\nu \in \mathbb{N}^n_0} 2^{-|\nu|/2} e^{-\|x\|^2/2} H_\nu(x)\varphi_\nu(z) {w^{|\nu|}} = \sum_{\nu \in \mathbb{N}^n_0} h_\nu(x) \varphi_\nu(z) {w^{|\nu|}}\label{1H}.
	\end{align*}
\end{proof}

For $0 < \epsilon < 1$, we define the following class of self-adjoint operators on $L^2_k(\mathbb{R}^n)$:
\begin{equation}\label{1F}
       \gamma_\epsilon(f)(x) := \left(\mathcal{C}_k(f)(\epsilon z), F_{\epsilon z}(x)\right)_k = \left(\int_{\mathbb{R}^n} F_{\epsilon z}(y)f(y) dw_k(y),F_{\epsilon z}(x)\right)_k.
\end{equation}
Let $f \in L^2_k(\mathbb{R}^n)$. A simple calculation shows that
$$\|\gamma_\epsilon(f)\|^2_{L^2_k} = \sum_{\nu \in \mathbb{N}^n_0} \left|\int_{\mathbb{R}^n}f(x) h_\nu(x) dw_k(x)\right|^2 \epsilon^{2|\nu|} \leq \|f\|^2_{L^2_k},$$
thus
\begin{align}
	\|\gamma_\epsilon\|_{\mathcal{G}^\infty} = \|\gamma_\epsilon\|_{L^2_k\rightarrow L^2_k} \leq 1.
\end{align}

 The density of the operator $\gamma_\epsilon$ can be computed by
\begin{align}\label{gfq}
	\rho_{\gamma_\epsilon}(x) = (F_{\epsilon z}(x), F_{\epsilon z}(x))_k = e^{-\|x\|^2}\sum_{\nu \in \mathbb{N}^n_0} \frac{H_\nu(x) H_\nu(x)}{2^{|\nu|}} \epsilon^{2|\nu|},
\end{align}
as the set $\{\varphi_\nu : \nu \in \mathbb{N}^n_0\}$ is orthonormal in $\mathcal{A}_k$.
Here, the density function $\rho_{\gamma} : \mathbb{R}^n \rightarrow \mathbb{R}$ of an operator $\gamma$ is formally defined by $\rho_{\gamma}(x) = \gamma(x,x)$, where $\gamma(x,y)$ denotes the integral kernel of $\gamma$.
Applying the Mehler-formula (\ref{1E}) in (\ref{gfq}), we get
\begin{align*}
 \rho_{\gamma_\epsilon}(x) = \frac{1}{(1-\epsilon^4)^{\gamma + n/2}} e^{- \frac{1 + \epsilon^4 }{1-\epsilon^4} \|x\|^2 } E_k\left(\frac{2\epsilon^2 x}{1-\epsilon ^4}, x \right).\label{1N}
\end{align*}
By (\ref{1C}) the trace norm of $\gamma_\epsilon$ is finite:
\begin{equation}\label{1N}
		\|\gamma_\epsilon\|_{\mathcal{G}^1} = Tr(\gamma_\epsilon) = \frac{1}{(1-\epsilon^4)^{\gamma + n/2}} \int_{\mathbb{R}^n}e^{- \frac{1 + \epsilon^4 }{1-\epsilon^4} \|x\|^2 } E_k\left(\frac{2\epsilon^2 x}{1-\epsilon ^4}, x \right) dw_k(x) < \infty.
\end{equation}
By H\"older's inequality in Schatten spaces, we deduce that  $\gamma_\epsilon \in \mathcal{G}^r(L^2_k(\mathbb{R}^n)) $ for all $1 \leq r \leq \infty$ and
\begin{equation}\label{1J}
	\|\gamma_\epsilon\|_{\mathcal{G}^r} \leq Tr(\gamma_\epsilon)^{\frac{1}{r}}.
\end{equation}
Such inequality (\ref{1J}) is known as Berezin-Lib inequality.
\subsection{Image of coherent states under the Schr\"odinger semigroup}
In order to calculate the image of $F_z$ under the Dunkl-Hermite Schr\"odinger semigroup we first generalize (\ref{1K}) involving the function $e^{-\delta \|x\|^2}$ for $\delta \in \mathbb{C}$ and obtain the following lemma.
\begin{lemma}
		Let $k \geq 0$ and  $\delta, z, w \in \mathbb{C}^n$ such that $Re(\delta) > 0$. Then
	\begin{equation}\label{1L}
		\int_{\mathbb{R}^n} e^{-\delta \|x\|^2} E_k(x,z) E_k(x,w) w_k(x)dx = \frac{c_k}{(2 \delta)^{\gamma + \frac{n}{2}}} e^{\frac{\ell(z) + \ell(w)}{4 \delta}} E_k\left(\frac{z}{2 \delta},w\right).
	\end{equation}
Here we denote $z^a = e^{a \hspace{1pt}{Ln}(z)}$ for $z = x + iy  (\{(x,y) \in \mathbb{R}^2 : x > 0\})$, where $Ln(z)$ is the principal branch of the complex logarithm.
\end{lemma}
\begin{proof}
	If $\delta \in \mathbb{R}$ and $\delta > 0$, using the change of variable $x \mapsto \frac{1}{\sqrt{2 \delta}} x$ in (\ref{1K}), we obtain (\ref{1L}).
	
	For $\delta \in \mathbb{C}$, we consider
	\begin{equation}\label{mn}
		F(\delta ) = \int_{\mathbb{R}^n} e^{-\delta \|x\|^2} E_k(x,z) E_k(x,w) w_k(x)dx
	\end{equation}
	and
	\begin{equation}
		G(\delta) = \frac{c_k}{(2 \delta)^{\gamma + \frac{n}{2}}} e^{\frac{\ell(z) + \ell(w)}{4 \delta}} E_k\left(\frac{z}{2 \delta},w\right).
	\end{equation}
	For $Re(\delta) > 0$, the integral in (\ref{mn}) converges and can be differentiated with respect to $\delta$, thus $F$ is analytic in $Re(\delta) > 0$. On the other hand $G$ is also analytic in $Re(\delta) > 0$ as $e^{\ell(z)}, E_k(z,w)$ are analytic in $z-$ variable ( see (\ref{nb})).
	
	Since $F$ and $G$ coincide on the positive real line, identity theorem gives
	$$F(\delta) = G(\delta),\quad \mbox{for all}~~~ Re(\delta) > 0. $$
\end{proof}
We show that the image of $F_z$ under the Dunkl-Hermite Schr\"odinger semigroup turns into another coherent state by a time dependent label change in the following proposition.
\begin{proposition}\label{Immage}
	Let $F_z(x) = e^{-(\ell(z)+\ell(x))/2} E_k(\sqrt{2}z,x), x\in \mathbb{R}^n$. Then
	\begin{equation}
		e^{-i t \mathcal{H}_k} F_z(x)= \frac{c_k^2}{(i e^{i t})^{\gamma + \frac{n}{2}}} F_{e^{it} z}(x).
	\end{equation}
\end{proposition}

\begin{proof}
	By (\ref{semigroup}) the image of $F_z$ under the Dunkl-Hermite Schr\"odinger semigroup can be computed as
	$$e^{-i t \mathcal{H}_k} F_z(x)=\frac{c_{k} e^ {-\frac{i}{2}\operatorname{cot} t \|x\|^{2}} e^{-\ell(z)/2}}{(i \sin t)^{\gamma + \frac{n}{2}}} \int_{\mathbb{R}^{n}} e^ {-(\frac{1}{2} + \frac{i}{2}\operatorname{cot} t) \|y\|^{2}} E_k \left( \frac{ix}{\sin t} , y \right)  E_k(\sqrt{2}z,y) w_{k}(y) d y.$$
	Using (\ref{1L}) we obtain
	\begin{align*}
		\int_{\mathbb{R}^{n}} e^ {-(\frac{1}{2} + \frac{i}{2}\operatorname{cot} t) \|y\|^{2}} E_k \left( \frac{ix}{\sin t} , y \right)  E_k(\sqrt{2}z,y) w_{k}(y) d y\\
		= \dfrac{c_k}{(1 + i \cot t)^{\gamma + \frac{n}{2}}} e^{\frac{\ell(z)}{1+i \cot t}} e^{-\frac{\ell(x)}{2 \sin^2 t (1 + i \cot t)}} E_k(\sqrt{2} e^{it}z,x).
	\end{align*}
After a simple calculation, we find that
$$e^{-\frac{\ell(x)}{2 \sin^2 t (1 + i \cot t)}} = e^{-\frac{\|x\|^2}{2}} e^{\frac{i}{2}\cot t\|x\|^2}\quad \mbox{and}\quad e^{\frac{\ell(z)}{1+i \cot t}} = e^{\frac{\ell(z)}{2}} e^{-\frac{\ell(e^{it}z)}{2}}$$
Thus
$$e^{-i t \mathcal{H}_k} F_z(x) = \frac{c_k^2}{(i e^{i t})^{\gamma + \frac{n}{2}}} e^{-\frac{\ell(e^{it}z)}{2}} e^{-\frac{\ell(x)}{2}}E_k(\sqrt{2} e^{it}z,x). $$
	
\end{proof}

\section{Necessary condition on the Schatten exponent}\label{sec4}
Using the results obtained in Section 3, we are in position to prove Theorem \ref{Thm1}.\\
\noindent{\bf \emph{Proof of Theorem \ref{Thm1}} : } It is enough to prove  the estimate (\ref{Schatten exponent1}) fails for all \(r>\frac{2 p n}{(1+p)n - (p-1)2\gamma}\) with $p,q$ in Theorem \ref{Thm1}.
The estimate (\ref{Schatten exponent1}) can also be written in terms of the operator
\begin{align}\label{CH2z}
	\gamma_0 :=\sum_{j} \lambda_{j}\left|u_{j}\right\rangle\left\langle u_{j}\right|
\end{align}
on $L^{2}_k\left(\mathbb{R}^{n}\right),$ where  the  Dirac's notation \(|u\rangle\langle v|\) stands  for the
rank-one operator $f \mapsto\langle v, f\rangle u$.  For such $\gamma_0$, let
$$\gamma(t) :=e^{-i t  \mathcal{H}_k} \gamma_0 e^{i t \mathcal{H}_k}=\sum_{j} \lambda_{j}\left|e^{-i t  \mathcal{H}_k} u_{j}\right\rangle\left\langle e^{-i t  \mathcal{H}_k} u_{j}\right|. $$
Then the density of the operator $\gamma(t)$ is given by
\begin{align}\label{CH2y}
	\rho_{\gamma(t)} :=\sum_{j} \lambda_{j}\left|e^{-i t  \mathcal{H}_k} u_{j}\right|^{2}.
\end{align}
With these notations (\ref{Schatten exponent1}) can be  rewritten  as
\begin{align}\label{CH27}
	\left\|\rho_{\gamma(t)}\right\|_{L^q ((-\frac{\pi}{2},\frac{\pi}{2}), L_{k}^p(\mathbb{R}^n))} \leq C_{n, q}\|\gamma_0\|_{\mathcal{G}^{\frac{2 p}{p+1}}},\end{align} where $\|\gamma_0\|_{\mathcal{G}^{\frac{2p}{p+1}}}=\left(\displaystyle\sum_j |\lambda_j|^{\frac{2p}{p+1}}\right)^{\frac{p+1}{2p}}$ and $\mathcal{G}^r$ is the Schatten $r$ class defined in Section \ref{sec2}.

In view of (\ref{CH27}), the proof of the Theorem \ref{Thm1} follows from the following proposition.

\begin{proposition}\label{6}  Suppose $k \geq 0$ and	$p, q, n \geqslant$ 1 satisfies
	$ 2 \gamma < \frac{n(p+1)}{p-1}$. Then we have
	\begin{align}\label{klq}
		\sup _{\gamma_0 \in \mathcal{G}^{r}} \frac{\left\|\rho_{e^{-i t\mathcal{H}_k} \gamma_0 e^{i t \mathcal{H}_k}}\right\|_{L^q ((-\frac{\pi}{2},\frac{\pi}{2}), L_{k}^p(\mathbb{R}^n))}}{\|\gamma_0\|_{\mathcal{G}^{r}}}=+\infty,
	\end{align}
	for all \(r>\frac{2 p n}{(1+p)n - (p-1)2\gamma}.\)
\end{proposition}

\begin{proof}
		For $1 <\epsilon < 1$, consider the operators $\gamma_\epsilon$ defined in (\ref{1F}), $$\gamma_\epsilon(f)(x) = \left(\int_{\mathbb{R}^n} F_{\epsilon z}(y)f(y) dw_k(y),F_{\epsilon z}(x)\right)_k.$$ After a simple computation and applying Proposition (\ref{Immage}) we obtain
	\begin{align*}
		\rho_{\gamma_\epsilon(t)}(x) :=& \rho_{e^{-i t\mathcal{H}_k} \gamma_\epsilon e^{i t \mathcal{H}_k}}(x)\\
		=& \left(e^{i t \mathcal{H}_k}F_{\epsilon z}(x), e^{i t \mathcal{H}_k}F_{\epsilon z}(x)\right)_k\\
		=& \left(F_{\epsilon e^{-it} z}(x), F_{\epsilon e^{-it} z}(x)\right)_k.
	\end{align*}
	By Lemma \ref{qa} we have $F_{\epsilon e^{-it} z}(x)  =  \sum_{\nu \in \mathbb{N}^n_0} h_\nu(x) \varphi_\nu(z) e^{-i|\nu|t}{\epsilon^{|\nu|}}$, thus
	\begin{align*}
		\rho_{\gamma_\epsilon(t)}(x) =& \sum_{\nu \in \mathbb{N}^n_0} h_\nu(x) h_\nu(x){\epsilon^{2|\nu|}}\\
		=& e^{-\|x\|^2}\sum_{\nu \in \mathbb{N}^n_0} \frac{H_\nu(x) H_\nu(x)}{2^{|\nu|}} \epsilon^{2|\nu|}\nonumber\\
		=& \frac{1}{(1-\epsilon^4)^{\gamma + n/2}} e^{- \frac{1 + \epsilon^4 }{1-\epsilon^4} \|x\|^2 } E_k\left(\frac{2\epsilon^2 x}{1-\epsilon ^4}, x \right).
	\end{align*}
 Therefore,
	\begin{align*}
		\|\rho_{\gamma(t)}\|_{L_k^p(\mathbb{R}^n)}^p &=\frac{c_k^{4p}}{(1-\epsilon^4)^{(\gamma + n/2)p}} \int_{\mathbb{R}^n} e^{- p\frac{1 + \epsilon^4 }{1-\epsilon^4} \|x\|^2 } E_k\left(\frac{2\epsilon^2 x}{1-\epsilon ^4}, x \right)^p dw_k(x)\\
		&=\frac{c_k^{4p}(1+\epsilon^2)^{(\gamma + \frac{n}{2})(1-p)}}{(1-\epsilon^2)^{(\gamma + \frac{n}{2})(p+1)}}\int_{\mathbb{R}^n} e^{- p \left(1 +\frac{2 \epsilon^2 }{(1-\epsilon^2)^2}\right)  \|x\|^2 } E_k\left(\frac{2\epsilon^2 x}{(1-\epsilon^2)^2}, x \right)^p dw_k(x),
	\end{align*}
	where the last equality is obtained using the change of variable $x \mapsto \sqrt{\frac{1+\epsilon^2}{1-\epsilon^2}}x$. So
	\begin{align}
		\nonumber\|\rho_{\gamma(t)}\|&_{L^q ((-\frac{\pi}{2},\frac{\pi}{2}), L_{k}^p(\mathbb{R}^n))}\\
		&~~~~~~~~~~~~~=\frac{\pi^{\frac{1}{q}} c_k^{4}(1+\epsilon^2)^{\frac{1-p}{p}(\gamma + \frac{n}{2})}}{(1-\epsilon^2)^{\frac{p+1}{p}(\gamma + \frac{n}{2})}}\left(\int_{\mathbb{R}^n} e^{- p \left(1 +\frac{2 \epsilon^2 }{(1-\epsilon^2)^2}\right)  \|x\|^2 } E_k\left(\frac{2\epsilon^2 x}{(1-\epsilon^2)^2}, x \right)^p dw_k(x)\right)^{\frac{1}{p}}.
	\end{align}
	Now from (\ref{j1}) and (\ref{j3}), we have $$e^{-\frac{1}{2\delta^2}\|x\|^2} E_k\left(\frac{x}{2\delta^2},x\right) \geq \frac{C^{-1} (\sqrt{2}\delta)^{2\gamma +n}}{w_k(B( x,  \delta))}, ~~~\mbox{for any}~~~ \delta > 0.$$ Hence
	\begin{align*}
		\|\rho_{\gamma(t)}\|&_{L^q ((-\frac{\pi}{2},\frac{\pi}{2}), L_{k}^p(\mathbb{R}^n))} \\\geq&\frac{ A_{n,p}}{(1-\epsilon^2)^{\frac{p+1}{p}(\gamma + \frac{n}{2})}}
		\left(\frac{1-\epsilon^2}{2 \epsilon}\right)^{2 \gamma } \left(\int_{\mathbb{R}^n}\frac{e^{-p\|x\|^2}}{\Pi_{\alpha \in \mathcal{R}^{+}}(| \langle \alpha,x \rangle| + \frac{1-\epsilon^2}{2\epsilon})^{2pk(\alpha)}}dw_k(x)\right)^{\frac{1}{p}} \\ \geq & \frac{ A_{n,p}}{(1-\epsilon^2)^{\frac{p+1}{p}(\gamma + \frac{n}{2})}}
		\left(\frac{1-\epsilon^2}{2 \epsilon}\right)^{2 \gamma } \left(\int_{\mathbb{R}^n}\frac{C^{-1}e^{-p\|x\|^2}}{\Pi_{\alpha \in \mathcal{R}^{+}}(| \langle \alpha,x \rangle| + \frac{1-\epsilon^2}{2\epsilon} + 1)^{2pk(\alpha)}}dw_k(x)\right)^{\frac{1}{p}}.
	\end{align*}
	Using (\ref{1J}) and (\ref{1N}) we can write
	\begin{align*}
		Tr(\gamma_\epsilon^r) \leq& \frac{1}{(1-\epsilon^4)^{\gamma + n/2}} \int_{\mathbb{R}^n}  e^{- \frac{1 + \epsilon^4 }{1-\epsilon^4} \|x\|^2 } E_k\left(\frac{2\epsilon^2 x}{1-\epsilon ^4}, x \right) dw_k(x)\\
		=& \frac{1}{(1-\epsilon^2)^{2\gamma + n}}\int_{\mathbb{R}^n} e^{- \left(1 +\frac{2 \epsilon^2 }{(1-\epsilon^2)^2}\right)  \|x\|^2 } E_k\left(\frac{2\epsilon^2 x}{(1-\epsilon^2)^2}, x \right) dw_k(x)\\
		\leq &\frac{1}{(1-\epsilon^2)^{2\gamma + n}} \left(\frac{1-\epsilon^2}{2 \epsilon}\right)^{2 \gamma } \int_{\mathbb{R}^n}\frac{Ce^{-\|x\|^2}}{\Pi_{\alpha \in \mathcal{R}^{+}}(| \langle \alpha,x \rangle| + \frac{1-\epsilon^2}{2\epsilon})^{2k(\alpha)}}dw_k(x)\\ \leq & \frac{C}{(1-\epsilon^2)^{ n}}\int_{\mathbb{R}^n}  e^{-\|x\|^2} dx.
	\end{align*}
	Thus
	\begin{align*}
		&\frac{\left\|\rho_{e^{-i t\mathcal{H}_k} \gamma_\epsilon e^{i t \mathcal{H}_k}}\right\|_{L^q ((-\frac{\pi}{2},\frac{\pi}{2}), L_{k}^p(\mathbb{R}^n))}}{\|\gamma_\epsilon\|_{\mathcal{G}^{r}}}\\ \geq& A_{n,p} \left(\frac{1}{1-\epsilon^2}\right)^{n(\frac{(1+p)n + (1-p)2\gamma}{2pn}-\frac{1}{r})}
		\left(\int_{\mathbb{R}^n}\frac{C^{-1}e^{-p\|x\|^2}}{\Pi_{\alpha \in \mathcal{R}^{+}}(| \langle \alpha,x \rangle| + \frac{1-\epsilon^2}{2\epsilon} + 1)^{2pk(\alpha)}}dw_k(x)\right)^{\frac{1}{p}}.
	\end{align*}
	Since
	\begin{align*}
		0 <\int_{\mathbb{R}^n}\frac{C^{-1}e^{-p\|x\|^2}}{\Pi_{\alpha \in \mathcal{R}^{+}}(| \langle \alpha,x \rangle| + 1)^{2pk(\alpha)}}dw_k(x) \leq \int_{\mathbb{R}^n} C^{-1} e^{-p\|x\|^2} dw_k(x),
	\end{align*}
	letting $\epsilon \rightarrow 1^-$ we get (\ref{klq}) for $r>\frac{2 p n}{(1+p)n + (1-p)2\gamma}$.
\end{proof}

\noindent{\bf \emph{Proof of Theorem \ref{thm2}} : } In view of Theorem \ref{Thm1} the proof follows directly from Lemma \ref{ed} and the elementary fact that the orthonormality of $(f_j)_j$ is preserved under complex conjugation.
\vspace{.02cm}
~~~~~~~~~~~~~~~~~~~~~~~~$\hfill\square$ \\
Now we proceed to the prove Theorem \ref{thm3}.\\
\noindent{\bf \emph{Proof of Theorem \ref{thm3}} : } Let $p,q \geq1$ such that $1 < p<\frac{2\gamma+n+ 1}{  2\gamma+n-1}$ {and} $	\frac{2}{q}+\frac{2 \gamma+n}{ p} \geq 2 \gamma+n.$ Then there exists a $0<\beta \leq 1$ such that $\frac{2\beta}{q} + \frac{2\gamma + n}{p} =  2 \gamma+n$. Since $\frac{q}{\beta} \geq q \geq 1$, by the inclusion relation of $L^q(-\frac{\pi}{2},\frac{\pi}{2})$- spaces, we have
$$\left\|\sum_{j} \lambda_{j }\left|e^{-i t \mathcal{H}_k}f_{j }\right|^{2}\right\|_{L^q ((-\frac{\pi}{2},\frac{\pi}{2}), L^p(\mathbb{R}^n))} \leq \left\|\sum_{j} \lambda_{j }\left|e^{-i t \mathcal{H}_k}f_{j }\right|^{2}\right\|_{L^{\frac{q}{\beta}} ((-\frac{\pi}{2},\frac{\pi}{2}), L^p(\mathbb{R}^n))}.$$
Clearly the pair $p, \frac{q}{\beta}$ satisfies the conditions of Theorem $\ref{THM2}$, so the first part of the Theorem \ref{thm3} follows from Theorem \ref{THM2}.

Note that $\frac{(1+p)n + (1-p)2\gamma}{2 p n} = 1-\frac{\beta}{nq} > 0$, thus the second part of the Theorem \ref{thm3} follows from Theorem \ref{Thm1}.
$\hfill\square$

\begin{remark}

	\begin{enumerate}
			\item Since, for $k \equiv 0$, we have $\frac{2 p n}{(1+p)n + (1-p)2\gamma}= \frac{2p}{p+1}$, the optimality of Schatten exponent in (\ref{Schatten exponentA}) and (\ref{Schatten exponentB}) also follows from Theorems \ref{thm2} and \ref{Thm1} as a special case.
			\item For $k\equiv 0$, the operator $\gamma_\epsilon$ coincides with the operator defined in the proof of Proposition 5.1 in \cite{shyam} (also in the proof of Proposition 1 in \cite{frank}) with $L^2=\mu=\frac{\epsilon^2}{1-\epsilon^2}, \beta=\frac{1}{2}$ and $z=\frac{x+i\xi}{\sqrt{2}}$.
	\end{enumerate}

\end{remark}

\section*{Acknowledgments}
The first author wishes to thank the Ministry of Human Resource Development, India for the  research fellowship and Indian Institute of Technology Guwahati, India for the support provided during the period of this work.


\begin{thebibliography}{99}
	\normalsize
	\baselineskip=17pt
	
\bibitem{Ank} J.-Ph. Anker, J. Dziuba\'{n}ski, A. Hejna, Harmonic functions, conjugate harmonic functions and the Hardy
space $H^1$ in the rational Dunkl setting, \emph{J. Fourier Anal. Appl.} 25, 2356--2418 (2019).	



\bibitem{ben}S. Ben Sa\"id, T. Kobayashi and B. Orsted, Laguerre semigroup and Dunkl operators, \emph{Compos. Math.} 148, 1265--1336 (2012).





\bibitem{Ratna3}  S. Ben Sa\"id, A. K. Nandakumaran and P. K. Ratnakumar, Schr\"odinger propagator and the Dunkl Laplacian,   \emph{hal-00578446v1} (2011).
 

\bibitem{BBen}	 	S. Ben Sa\"id and B. \O{}rsted, Segal--Bargmann transforms associated with finite Coxeter groups, \emph{ Math. Ann.} 334, 281--323 (2006).








\bibitem{lee} N. Bez, Y. Hong, S. Lee, S.  Nakamura and Y. Sawano,    {On the Strichartz estimates for orthonormal systems of initial data with regularity}, \emph{Adv. Math.}   {354}, Paper No. 106736 (2019).



\bibitem{BEZ} N.  Bez, Neal, S.  Lee and S.  Nakamura,  Strichartz estimates for orthonormal families of initial data and weighted oscillatory integral estimates,  \emph{Forum Math. Sigma} 9, Paper No. e1  (2021).

 \bibitem{Lwin1} T. Chen, Y. Hong and N. Pavlovi\'c, Global well-posedness of the NLS system for infinitely many fermions, \emph{Arch. Ration. Mech. Anal.} 224, 91--123 (2017).


\bibitem{Lwin2}  T. Chen, Y. Hong, N. Pavlovi\'c, On the scattering problem for infinitely many fermions in dimension $d \geq 3$ at positive temperature, \emph{Ann. Inst. H. Poincar\'e{} C Anal. Non Lin\'eaire} 35, 393--416 (2018).





\bibitem{dun}  	C. F. Dunkl, Differential-difference operators associated to reflection groups, \emph{Trans. Amer. Math. Soc.} 311, 167--183 (1989).


\bibitem{Dunklll}	C. F. Dunkl, Hankel transforms associated to finite reflection groups, \emph{Contemp. Math.} 138,123--138 (1992).







\bibitem{frank}  R. L. Frank, M. Lewin,  E. H. Lieb and   R. Seiringer, Strichartz inequality for orthonormal functions, \emph{J. Eur. Math. Soc. (JEMS)} 16(7), 1507--1526 (2014).


\bibitem{FS} R. L. Frank and J. Sabin, Restriction theorems for orthonormal functions, Strichartz inequalities, and uniform Sobolev estimates, \emph{Amer. J. Math.} 139(6), 1649--1691 (2017).

\bibitem{GG}   S. Ghazouani, {Coherent states of the one-dimensional Dunkl oscillator for real and complex variables and the Segal\textendash{}Bargmann transformation of Dunkl-type}, \emph{J. Phys. A} 55(50), Paper no. 505203 (2022).

\bibitem{shyam1} S. Ghosh, S. S. Mondal and J. Swain, Strichartz inequality for orthonormal functions associated with special Hermite operator,   \emph{Forum Math.} (2023).



\bibitem{hum} 	J. E. Humphreys, {Reflection Groups and Coxeter Groups}, Cambridge Stud. Adv. Math. 29, Cambridge University Press, Cambridge (1990).


\bibitem{Lewin} M. Lewin and J. Sabin, The Hartree equation for infinitely many particles. II. Dispersion and scattering in 2D, \emph{Analysis and PDE} 7, 1339--1363 (2014).


\bibitem{lieeb} E. H. Lieb, The stability of matter: from atoms to stars, \emph{Bull. Amer. Math. Soc.} 22, 1--49 (1990).

\bibitem{lieb}  E. H. Lieb and   W. E. Thrring, Bound on kinetic energy of fermions which proves stability of matter,  \emph{Phys. Rev. Lett.} 35, 687--689 (1975).

\bibitem{liebb}  E. H. Lieb and   W. E. Thrring, {Inequalities for the moments of the eigenvalues of the Schr\"odinger hamiltonian and their relation to Sobolev inequalities}, Studies in Mathematical Physics, Princeton university press, page 269-303, Princeton (1976).



\bibitem{SSM} S. S. Mondal and M. Song, Orthonormal Strichartz inequalities for the (k,a)-generalized Laguerre operator and Dunkl operator, to apear in \emph{Israel J. Math.} (2023). arXiv:2208.12015



\bibitem{shyam} 	S. S. Mondal and J. Swain, Restriction theorem for the Fourier-Hermite transform and solution of the Hermite-Schr\"odinger equation, \emph{Adv. Oper. Theory} 7(4), Paper No. 44 (2022).



\bibitem{nakamura}	S. Nakamura, The orthonormal  Strichrtz inequality on Torus, \emph{Trans. Amer. Math. Soc.} 373, 1455-1476 (2020).

\bibitem{Rost} 	 	R. R\"{o}sler,  Generalized Hermite polynomials and the heat equation for Dunkl operators, \emph{Comm. Math. Phys.} 192, 519-542 (1998).







\bibitem{PRA} P. J. K. Senapati, B. Pradeep, S. S. Mondal and H. Mejjaoli, Restriction Theorem for the Fourier-Dunkl Transform and Its Applications to Strichartz Inequalities, \emph{J. Geom. Anal.} 34(3), Paper No. 74 (2024).

\bibitem{Simon} B. Simon, Harmonic analysis. A Comprehensive Course in Analysis, Part 3. American Mathematical Society, Providence, RI, 2015. xviii+759 pp. ISBN: 978-1-4704-1102-2 (2015).

\bibitem{Soltani} F. Soltani, Generalized Fock spaces and Weyl commutation relations for the Dunkl kernel. \emph{Pac. J. Math.} 214, 379--397 (2004).















\end{thebibliography}
\end{document}